\documentclass[10pt]{amsart}
\usepackage{graphicx}
 \usepackage{caption}
\usepackage{psfrag}
\usepackage{float}
\usepackage{subfigure}
\usepackage{psfrag}
\usepackage{epsfig}
\usepackage{epstopdf}
\usepackage{amsmath}
\usepackage{amscd}

\usepackage{amsthm,amsfonts,amsxtra,amssymb,amscd}
\usepackage[all]{xy}
\xyoption{2cell}
\usepackage{enumerate}

\usepackage{pdfsync}

 \usepackage[unicode]{hyperref}

\font\teneufm=eufm10
\font\seveneufm=eufm7
\font\fiveeufm=eufm5
\newfam\eufmfam
\textfont\eufmfam=\teneufm
\scriptfont\eufmfam=\seveneufm
\scriptscriptfont\eufmfam=\fiveeufm

\theoremstyle{plain}
\newtheorem*{lemma*}{Lemma}

\newtheorem*{theorem*}{Theorem}
\newtheorem{theorem}{Theorem}
\newtheorem*{proposition*}{Proposition}

\newtheorem*{corollary*}{Corollary}

\newtheorem{corollary}[subsection]{Corollary}
\theoremstyle{definition}
\newtheorem*{definition*}{Definition}

\newtheorem*{example*}{Example}

\theoremstyle{remark}
\newtheorem*{remark*}{Remark}\newtheorem*{proof*}{proof}

\newcommand{\be}{\begin{equation}}
\newcommand{\ee}{\end{equation}}

 \title[Invariants of commuting matrices]{Invariants of commuting matrices  }
\author{  
 C. Procesi  }  \thanks{I thank Pavel Etingof  for pointing out this problem and some discussions.}
 
\begin{document}
\maketitle
\section{An open question}\subsection{Introduction}  Consider the space of pairs $(X,Y)$ of $n\times n$  matrices over a field $F$ of characteristic 0, and denote by  $R_n(X,Y)$ the polynomial algebra  on this space.

An intriguing, difficult open question in commutative algebra asks wether the ideal  $I_n$, generated by the entries of the equation $[X,Y]=XY-YX=0$  is a prime ideal in $R_n(X,Y)$. Since it is known that the variety of pairs of commuting matrices is irreducible (cf. \cite{GE}) this means that the ideal of this variety is generated by these natural quadratic equations.

Several attempts to this have been made but, as far as I understand this issue, no conclusive proof has been reached except for small $n\leq 4$ (Wallach private communication).  In \cite{GG}  the authors study a  related problem, that is they study pairs of matrices with commutator of rank 1.  

From this they deduce that, for the case of 2 commuting matrices,  the  ring of invariants  of the natural action by conjugation  on  $R_n(X,Y)/I_n$  is an integral domain.   The purpose of this paper is to comment two  papers, one by Domokos \cite{Do} and one of Francesco Vaccarino \cite{vac} where they prove  the same statement but for the algebra of $m$--commuting matrices, for all $m$. This is in fact a consequence of a stronger theorem \ref{main}, which gives an isomorphism of this ring of invariants with the invariants of diagonal matrices under the symmetric group. The method of proof in the two papers is quite different.  The two papers are clearly independent and although published with some time gap I have been informed that the two preprints appeared on the archive  more or less at the same time.\smallskip

The approach of Vaccarino is through polynomial maps  and the one of Domokos  via classical invariant theory.

I  want to point out some simply related approach which may apply  to classical Lie algebras, since the same question can be asked for all simple Lie algebras (cf. \cite{Ri}, \cite{Te} and \cite{W})  but I have made no  attempt in this direction. It also gives some supplementary information on these algebras.
\subsection{Commuting matrices}
 We work over a field $F$ of characteristic 0, for instance $\mathbb Q$.  Let us fix a positive integer $n$, for all $m$ in fact for $m=\infty$  let us consider the ring of   polynomials in the entries $x^i_{h,k},\ h,k=1,\ldots,n,\ i=1,\ldots, m$ of $m$, $n\times n$ matrix variables $X_i:=(x^i_{h,k})$, modulo the ideal generated by the entries of the   commuting equations $[X_i,X_j]=0$. Denote by $A_{n,m}$ this ring and $A_n=A_{n,\infty}$.  On this ring acts the linear group $GL(m,F)$  by linear transformations on the variables $X_i$ and the group $GL(n,F)$ by conjugation; the two actions commute.
 
 One may consider  in the ring $M_n(A_{n,m})$  the $m$ generic commuting matrices $\xi_i$ images of  the  generic matrices $X_i$. \smallskip

 Consider now the  ring $B_{n,m}$ ($B_m=B_{n,\infty}$) of polynomials in $m$ vector variables $x_i=(x_{i,1},\ldots,x_{i,n})$  which we also consider as  $m$ diagonal matrices.  On this ring acts the linear group $GL(m,F)$  by linear transformations on the variables $x_i$ and the symmetric group $S_n$ by conjugation; the two actions commute. We can also identify $B_{n,m}= B_{n}^{\otimes m}$ where $B_n=F[x_1,\ldots,x_n]$ is the polynomial ring in $n$ variables.
 
 We have a natural {\em restriction}  map  $\pi: A_{n,m}\to B_{n,m}$  compatible under the action  of  $GL(m,F)$ and $S_n$, which induces thus  a natural map  of invariants
 \begin{equation}\label{comi}
\tilde\pi: A_{n,m}^{GL(n,F)}\to B_{n,m}^{S_n}=(F[x_1,\ldots,x_m]^{\otimes n})^{S_n}.
\end{equation}
\begin{theorem}\label{main}[Domokos,  Vaccarino] The map $\tilde\pi$  is an isomorphism.

\end{theorem}
\begin{proof}  The proof from \cite{Do} is similar to the proof I want to present, the one in \cite{vac} is based on the Theory of polynomial maps.

One first considers the isomorphism $F[x_1,\ldots,x_m]\to F[\xi_1,\ldots,\xi_m]$ to the ring  of generic commuting matrices,  when we compose this map with the determinant  we have a polynomial map, homogeneous of degree $n$: 
\begin{equation}\label{detm}
D:F[x_1,\ldots,x_m]\to F[\xi_1,\ldots,\xi_m]\stackrel{\det}\longrightarrow   A_{n,m}^{GL(n,F)}
\end{equation} This map is {\em multiplicative} in the sense that $D(ab)=D(a)D(b)$, hence by a general theorem of Roby,  \cite{Roby1} it factors through a homomorphism of the $n^{th}$ symmetric tensors  $\bar D:(F[x_1,\ldots,x_m]^{\otimes n})^{S_n}\to A_{n,m}^{GL(n,F)}$. So the point is  just to prove that this map is inverse to the restriction. The composition $\bar\pi\circ \bar D$ is clearly identity, so $\bar D$ is injective and $\bar \pi$ surjective the issue is in the other order  $\bar D\circ \pi$.                                                           In \cite{vac}  this is proved by showing a correspondence between generators of the two algebras.  For matrix invariants it is known (even  before imposing the commutative law), that invariants are generated by the coefficients of the characteristic polynomials of primitive monomials in the variables $X_i$, see \cite{Am}, a similar statement holds for the symmetric group  and one verifies the correspondence.

Here I want to point out  a slightly different approach, which  consists in applying the classical Arhonold method, of reducing to multilinear elements. As I will point out at the end this also gives some normal form for the invariants.
Remark that $\tilde\pi$ is  $GL(m,F)$  equivariant for all $m$, thus it is enough to prove the statement for $m=\infty$ and, by the classical method of Arhonold, it is enough to prove that $\tilde\pi$ is an isomorphism when restricted to  the multilinear elements.

We thus have to gather some information on these elements.  

Let us start with the multilinear elements of  $B_{n,m}^{S_n}$.  The multilinear elements  of $B_{n,m}$ have as basis the monomials
$$x_{1,i_1}\ldots x_{m,i_n};\quad 1\leq i_j\leq n,\ \forall j.$$
Thus these monomials can be thought of as maps  from $[m]:=\{1,2,\ldots,m\}$ to $[n]$.

These monomials are permuted by the symmetric group so that the space of  multilinear elements of  $B_{n,m}^{S_n}$ has as basis the orbits of $S_n$ on such a space of functions.  

To any function $f:[m]\to  [n]$, is associated the partition of  $[m]$ into its fibers  which is a partition of   $[m]$ into at most $n$ subsets.  This partition is independent of the orbit.  Conversely given 
 a partition $\Lambda$ of   $[m]$ into at most $n$ subsets, we have a natural ordering of these sets  by associating to each subset its minimal element  and then ordering them according to the minimal element as $S_1, S_2,\ldots,S_k$. We then associate to this the function $f_\Lambda:[m]\to  [n]$ which takes the value $i$ on $S_i$. In its orbit under $S_n$ this function is {\em leading} in a suitable lexicographic order.  Hence it is easily seen that in this way we have established a 1--1 correspondence between maps $f:[m]\to  [n]$ up to symmetry and decompositions $\Lambda$ of $[m]$ into at most $n$ subsets.  
 
 \medskip
 
 Let us now turn our attention to  $A_{n,m}^{GL(n,F)}$. By classical invariant theory  multilinear invariants of  $m$ matrix variables  are spanned by the functions $\phi_\sigma$ where $\sigma\in S_m$  is a permutation and 
 $$\phi_\sigma(x_1,x_2,\dots,x_m)=tr(x_{i_1}x_{i_2}\dots
x_{i_h})tr(x_{j_1}x_{j_2}\dots x_{j_k})\dots tr(x_{s_1}x_{s_2}\dots x_{s_m})$$ there the various factors correspond to the decomposition into cycles of $\sigma$ (see for instance \cite{P7}).

When we are working with $n\times n$ matrices we have a basic formal identity  (equivalent to the Cayley--Hamilton identity)
$$ \sum_{\sigma\in S_{n+1}}\epsilon_\sigma \phi_\sigma(x_1,x_2,\dots,x_{n+1})=0.$$
We deduce from this  that  $$\phi_1(x_1,x_2,\dots,x_{n+1})=\prod_{i=1}^{n+1}tr(x_i)=- \sum_{\sigma\in S_{n+1},\ \sigma\neq 1}\epsilon_\sigma \phi_\sigma(x_1,x_2,\dots,x_{n+1})$$ and then each summand of the R.H.S. of this equality is the product of at most $n$ factors of the form $tr(M)$  for $M$ some monomial.\smallskip

Therefore  by a repeated use of this identity we deduce that every   invariant for $n\times n$ matrices is a linear combination  of   products with at most  $n$ factors of the form $tr(M)$  for $M$ some monomial.\smallskip

We apply this in particular to the multilinear invariants  for commuting matrices, in this case   the invariant $tr(M)$ for a multilinear monomial $M$  depends only on the variables appearing in $M$ and not the order in which they appear.

Given a set $S$ of indices  we may define thus an element $tr_S\in A_n$
$$tr_S:= tr(\prod_{i\in S}X_i).$$
Thus, given a  partition $\Lambda:=\{S_1,\ldots,S_k\},\ k\leq n$ of   $[m]$ into at most $n$ subsets $S_i$  and considering the  invariant  $t_\Lambda:=\prod_it_{S_i}$  we have that these multilinear elements span linearly all multilinear invariants.  Thus we can prove that they form a basis and that $\tilde\pi$ is an isomorphism as soon as we show that $\tilde\pi$ is surjective (proved already in \cite{W}).\smallskip

For this order  lexicographically the monomials  $x_{1,i_1}\ldots x_{m,i_n}$ which can  be just thought of as words of length $m$ in the letters $1,2,\ldots,n$.  Observe that the leading term of  $t_\Lambda$ is the monomial associated to $f_\Lambda$  and the claim follows.
\end{proof}
\smallskip

Let us draw some consequences, let $J$ be the nil ideal of  $A_n$ clearly it is stable under both groups  and we have
\begin{corollary}
$J$  does not contain any  element which is invariant under  $GL(n,F)$.
\end{corollary}

We should notice that if one is interested in describing these algebra,  in detail, one method is to describe them as representation of $GL(m)$ ($m$ is the number of copies) and this is done by describing each homogeneous component of degree $k$ as a direct sum of Schur functors $S_\lambda(F^k)$ where $\lambda\vdash k$ is a partition of $k$.

 The description of homogeneous component of degree $k$ passes through the description of  the multilinear part of   degree $k$ for the case of $k$ copies as representation of the symmetric group  implies the description of  the homogeneous part of the same degree as representation  of the linear group  $GL(k)$.

In fact  by the Theory of Schur functors,  a partition $\lambda\vdash k$  defines a Schur functor  $S_\lambda(W)$ on vector spaces homogeneous of degree $k$.  When we apply this to an  $k$--dimensional space $W=F^k$  with a prescribed basis  we may consider inside  the space $M_\lambda(F^k)$ of multilinear elements, which by definition are the  invariants for the torus of  diagonal matrices  in that basis with determinant 1. It then is true that $M_\lambda(F^k)$  is the irreducible representation of the symmetric group $S_k$, permuting that basis, associated to the same partition $\lambda$.

Of course  $S_\lambda(W)=0$  if the dimension of $W$ is strictly smaller than the height of $\lambda$ so when $m<k$ not all the representations appearing contribute.  In our case it would remain  to describe the decomposition into irreducible representations  of the permutation representations, which we have seen describe the multilinear invariants, of decompositions $\Lambda$ of $[m]$ into at most $n$ subsets. The orbits  of this permutation action  are parametrized by partitions of $m$ of height $\leq n$, to each such partition corresponds the permutation representations given on cosets of the corresponding Young subgroup.  These are well understood,    there is an extensive literature, see for instance \cite{Sagan}.

 \vskip5pt
\footnotesize{
\vskip20pt
Dipartimento di Matematica, Sapienza Universit\`a di Roma, P.le A. Moro 2,
00185, Roma, Italy; 
\par\noindent
Email address:
\par\noindent{\tt procesi@mat.uniroma1.it}
}

 \end{document}